\documentclass[a4paper,reqno,12pt]{amsart}
\usepackage{verbatim}
\usepackage{amssymb,amsmath,amsthm}
\usepackage{amsfonts}
\usepackage{array}
\usepackage{xcolor}
\usepackage{hyperref}
\usepackage[enableskew]{youngtab}
\usepackage{ytableau,varwidth}

\begin{document}
    
\newtheorem{theorem}{Theorem}[section]
\newtheorem{lemma}[theorem]{Lemma}
\newtheorem{corollary}[theorem]{Corollary}
\newtheorem{conjecture}[theorem]{Conjecture}
\newtheorem{proposition}[theorem]{Proposition}

\newtheorem{definition}[theorem]{Definition}
\newtheorem{example}[theorem]{Example}
\newtheorem{xca}[theorem]{Exercise}

\definecolor{red}{rgb}{1,0,0}
\definecolor{blue}{rgb}{.2,.2,.8}
\newcommand{\ccom}[1]{{\color{teal}{ #1}}}

\newtheoremstyle{remark}
    {\dimexpr\topsep/2\relax} 
    {\dimexpr\topsep/2\relax} 
    {}          
    {}          
    {\bfseries} 
    {.}         
    {.5em}      
    {}          

\theoremstyle{remark}
\newtheorem{remark}{Remark}[section]

\allowdisplaybreaks

\makeatletter
\@namedef{subjclassname@2020}{%
  \textup{2020} Mathematics Subject Classification}
\makeatother

\title{Lambert series and double Lambert series}

\begin{abstract}
We consider relationships between classical Lambert series, multiple Lambert series and classical $q$-series of the Rogers-Ramanujan type. We conclude with a contemplation on the Andrews-Dixit-Schultz-Yee conjecture.    
\end{abstract}

\author{T. Amdeberhan}
\address{Department of Mathematics\\ Tulane University\\ New Orleans, LA 70118, USA}
\email{tamdeber@tulane.edu}

\author{G. E. Andrews}
\address{Department of Mathematics\\ Penn State University\\ University Park, PA 16802, USA}
 \email{gea1@psu.edu} 

\author{C. Ballantine}
\address{Department of Mathematics and Computer Science\\ College of the Holy Cross \\ Worcester, MA 01610, USA \\} 
\email{cballant@holycross.edu} 


\subjclass[2020]{Primary 11P81; Secondary 05A17.}

\keywords{Lambert series, double Lambert series, Dirichlet series, Eisenstein series}

\maketitle

\section{introduction}

J.H. Lambert \cite{L} showed in 1771 that the generating function for the divisor function is given by
                          $$\sum_{n\geq1} \frac{q^n}{1 - q^n}.$$
More generally, we now define a Lambert series to be the series  given by 
$$\sum_{n\geq1} \frac{a_nq^n}{1-q^n}.$$
In particular, when $a_n=n^{2k-1}$, the resulting functions are either modular forms or (when $k=1$) a quasimodular form. A nice account of much that is known about Lambert series identities may be found in Schmidt \cite{S}.

While we shall consider classical Lambert series in this paper, we will also be extending our work to series of the form
$$\sum_{n\geq1} R_n(q^n,q),$$
where $R_n(x,y)$ is a rational function of $x$ and $y$. We call such series \emph{general Lambert series}. We note that our general Lambert series contain the generalized Lambert series considered by Schmidt \cite{S} and others.

In addition, we shall also explore series of the type
$$\sum_{n\geq1} S_{n,m}(q^n,q^m,q),$$
where $S_{n,m}(x,y,z)$ is a rational function of $x, y$ and $z$. We shall call such series \emph{double Lambert series}.

In Section 2, we gather background material. In Section 3, we prove identities connecting general and double Lambert series. In Section 4, we consider a connection of Lambert series to Rogers-Ramanujan type identities. It was previously shown \cite[eq'n (4.7)]{Andrews5} that
$$\sum_{n\geq1} \frac{(q;q)_n(q,q)_{n-1}q^{n^2}}{(q;q)_{2n}}
=\sum_{n\geq1}  \frac{(\frac{n}3)q^n}{1-q^n}$$
where $(\frac{n}p)$ is the Legendre symbol and 
$$(a;q)_n:=\prod_{j=0}^{n-1}(1-aq^j).$$
It is natural to ask if there are other such identities. To our surprise, there is a very similar result in which $(\frac{n}3)$ is replaced by $(\frac{n}5)$:
$$\sum_{n\geq1} \frac{(q;q)_n(q,q)_{n-1}(-1)^{n-1}q^{\binom{n+1}2}}{(q;q)_{2n}}
=\sum_{n\geq1}  \frac{(\frac{n}5)q^n}{1-q^n}.$$
Section 5 concludes with a careful albeit inconclusive consideration of the Andrews-Dixit-Schultz-Yee (ADSY) conjecture \cite{Andrews2} that
$$\sum_{m,n\geq1} \frac{q^{2mn+n}}{(1+q^n)(1-q^{2m-1})}$$
is an odd function of $q$. It would be a major step forward in the study of multiple Lambert series if one could develop methods that would not only prove this conjecture but also place it in the context that would reveal and prove similar results. While we have been unsuccessful in proving the ADSY conjecture, we believe our exploration points in the right direction.

\section{Background and notation}

A {\it partition}  of $n\in\mathbb{N}$ is a finite non-increasing sequence of positive integers that add up to $n$. We use the notation $\lambda=(\lambda_1,\lambda_2,\dots,\lambda_{\ell})$, with $\lambda_1\geq\lambda_2\geq \ldots \geq \lambda_\ell$ and 
$\lambda_1+\lambda_2+\cdots+\lambda_{\ell}=n$. We refer to the numbers $\lambda_i$ as the {\it parts} of $\lambda$. The number of parts in $\lambda$ is denoted by $\ell(\lambda)$ and it is called the \textit{length} of $\lambda$. We denote by $\mathcal P$ the set of all partitions.  For convenience, we set $\lambda_j=0$ if $j>\ell(\lambda)$.

Graphically, a partition $\lambda=(\lambda_1,\lambda_2,\dots,\lambda_{\ell})$ can be represented by a Young diagram, a left-justified array of boxes with $\lambda_i$ boxes in the $i$th row. For example, the Young diagram of $\lambda=(4,4,2,2,1)$ is 

$$\tiny\ydiagram{4,4,2,2,1}$$

The conjugate of a partition $\lambda$, denoted by $\lambda'$, is the partition whose Young diagram  is obtained from the Young diagram of $\lambda$ by reflection across the diagonal. For example, the conjugate of $\lambda=(4,4,2,2,1)$ is $\lambda'=(5,4,2,2)$. 

We also use the frequency notation for partitions. In particular, by $(k^j)$ we mean the partition of length $j$ with all parts equal to $k$.

The Frobenius symbol, introduced in 1900 by Frobenius \cite{Frobenius}, is defined as follows. Given a partition $\lambda$, let $r$ be the size of the largest square that fits inside the Young diagram of $\lambda$. This square is referred to as the Durfee square of $\lambda$. For example, the Durfee Square of $\lambda=(4,4,2,2,1)$ has size $2$. For $j$ such that $1\leq j\leq r$ let $a_j$, respectively $b_j$, be the number of boxes directly to the right, respectively below, the box in position $(j,j)$ in the Young diagram of $\lambda$. The Frobenius symbol of $\lambda$ is denoted by
$$(a_1,\dots,a_r\, \vert\,  b_1,\dots,b_r).$$
For example, the Frobenius symbol of $\lambda=(4,4,2,2,1)$ is $(3,2 \, \vert \, 4,2)$.

The Pochhammer symbol is defined by \begin{align*}
& (a;q)_n = \begin{cases}
1, & \text{for $n=0$,}\\
(1-a)(1-aq)\cdots(1-aq^{n-1}), &\text{for $n>0$;}
\end{cases}\\
& (a;q)_\infty = \lim_{n\to\infty} (a;q)_n.
\end{align*}
We assume $q$ is a complex number with $|q| < 1$ so that all infinite series and products converge absolutely.

\section{Three Lambert series} \label{S1}

\noindent
We begin by considering the following series
\begin{equation} \label{f1} f_1(q):=\sum_{k\geq1} \left(\frac{k(k-1)q^k}{1-q^k}-\frac{2kq^{2k}}{(1-q^k)^2}\right).\end{equation}
The next lemma gives the Fourier expansion of $f_1(q)$. Throughout the article, whenever we refer to divisors of a positive integer, we only consider the positive divisors. The algebraic proof of the lemma is straightforward. We give a combinatorial proof using integer partitions.

\begin{lemma}
$$f_1(q)=\sum_{n\geq1} \left(\sum_{d\vert n} (d^2+d-2n)\right)q^n.$$
\end{lemma}

\begin{proof}[Combinatorial proof]
Let $\mathcal B(n)$ be the set of partitions of $n$ with all parts equal (i.e., rectangular partitions) and let $\mathcal B^*(n)$ be the the set of partitions of $n$ with parts in two colors $1$ and $2$, all of the same size, and at least one part of each color. The order of colored positive integers is $1_1<1_2<2_1<2_2<\cdots $. 

Let $$a(n):=\sum_{\lambda\in \mathcal B(n)}\lambda_1(\lambda_1-1)$$ and $$b(n):=\sum_{\lambda\in \mathcal B^*(n)}2\lambda_1.$$
The function $f_1(q)$ defined by \eqref{f1} is the generating function for the sequence $a(n)-b(n)$. 

\smallskip
\noindent
To each partition $\lambda=(k^j)\in \mathcal B(n)$ we associate $j-1$ partitions in $\mathcal B^*(n)$. For each $1\leq i\leq j-1$ we color the first $i$ parts of $\lambda$ in color $2$ and the remaining $j-i$ parts in color $1$ to obtain a partition $\lambda^{(i)}\in \mathcal B^*$. Then $\lambda$ contributes $k(k-1)=k^2-k$ to $a(n)$ and the partitions $\lambda^{(i)}$, $1\leq i\leq j-1$ contribute $2k(j-1)=2(kj-k)=2n-2k$ to $b(n)$. Hence, each divisor $k$ of $n$ contributes $k^2-k-2n+2k=k^2+k-2n$ to $a(n)-b(n)$. 
\end{proof}

\bigskip
\noindent
Next, we introduce two new $q$-series that turn out to be equal to $f_1(q)$. Let
\begin{align*}
f_2(q)&:=\sum_{k\geq1}\sum_{\ell\geq1}
\left(\frac{(k+\ell)\,q^{k+\ell+k \ell}}{(1-q^k)(1-q^{\ell})}\right)\end{align*} and
\begin{align*} f_3(q)&:=\sum_{k\geq1}\sum_{\ell>k}
\left(\frac{2k\,q^{k+\ell}}{(1-q^k)(1-q^{\ell})}\right).
\end{align*}

We first prove, both analytically and combinatorially, that the two new series are equal to each other. 
\begin{theorem} We have $f_2(q)=f_3(q)$.
\end{theorem}
\begin{proof}[Algebraic proof] Rearranging and using properties of geometric series, we obtain
\begin{align*}  f_3(q) &= \sum_{\ell>k\geq1}\sum_{i,j\geq1} 2kq^{ki+\ell j} = \sum_{m,k\geq1}\sum_{i,j\geq1} 2kq^{ki+kj+mj}   \\
&  =\sum_{i,j,k,\ell\geq1} 2kq^{k\ell+ki+\ell j} = \sum_{i,j,k,\ell\geq1} (k+\ell)q^{k\ell+ki+\ell j}  = f_2(q).
\end{align*}
The penultimate equality is valid due to the symmetry in $k$ and $\ell$.
The proof is complete. 
\end{proof}

\begin{proof}[Combinatorial proof] First we interpret $f_2(q)$ and $f_3(q)$ as partition generating functions.

\noindent
Let $\mathcal C(n)$ be the set of partitions $\lambda$ of $n$ with exactly two different part sizes. For example $$\mathcal C(6)=\{(5,1), (4, 2), (4,1,1), (3, 1,1,1), (2, 2,1,1), (2, 1,1,1,1)\}.$$ 
Let $c(n)$ be the sum of smallest parts in all partitions in $\mathcal C(n)$ (without multiplicity) and let $d(n)$ be the sum of the multiplicities of the largest part in all partitions in $\mathcal C(n)$. In the above example, $c(6)=1+2+1+1+1+1= 7$ and $d(6)=1+1+1+1+2+1=7$. 

\smallskip
\noindent
The function $$f_2(q)=\sum_{k\geq1}\sum_{\ell\geq1}
\left((k+\ell)\, q^{kl}\frac{q^{k}}{1-q^k}\frac{q^{\ell}}{1-q^{\ell}}\right)$$ is the generating function for the sequence $c(n)+d(n)$  
and the function  $$f_3(q)=\sum_{k\geq1}\sum_{\ell>k}
\left(2k\, \frac{q^{k}}{1-q^k}\frac{q^{\ell}}{1-q^{\ell}}\right)$$ is the generating function for the sequence $2c(n)$.  

\smallskip

\noindent
Since the conjugate of any partition in $\mathcal C(n)$ is in $\mathcal C(n)$, and by conjugation $c(n)=d(n)$, it follows that $f_2(n)=f_3(n)$. 
\end{proof}

Next, we prove that the new $q$-series $f_2(q)$ and $f_3(q)$ are equal to $f_1(q)$. 
First we prove a helpful lemma. 

\begin{lemma} \label{lemma_f1_f3} We have \begin{equation}\label{Liouville} \sum_{k\geq1}\sum_{\ell>k} \frac{k q^{\ell}}{(1-q^k)(1-q^{\ell})}= \sum_{k\geq1} \frac{k^2 q^k}{1-q^k} - \sum_{k\geq1} \frac{kq^k}{(1-q^k)^2}.\end{equation}
\end{lemma}
\begin{proof}  It follows from \cite[equation (2.1)]{Andrews} that the right hand side of \eqref{Liouville} equals \begin{equation}\label{LR}\sum_{\substack{x_1,x_2,x_4,x_5\geq 1,  {x_3}\geq0}} 
q^{x_1x_2+x_2{x_3}+{x_3}x_4+x_4x_5}.\end{equation}
We show that the $q$-series \eqref{LR} is also equal to the left hand side of \eqref{Liouville}.

    From \cite[page 123]{Andrews}), 
\begin{align*}
{\sum_{\substack{x_1,x_2,x_4,x_5\geq 1, \\ {x_3}\geq0}} 
q^{x_1x_2+x_2{x_3}+{x_3}x_4+x_4x_5}}
&=\sum_{m,n\geq1} \frac{q^{2m+2n}}{(1-q^m)(1-q^n)(1-q^{m+n})} + \left(\sum_{k\geq1}\frac{q^k}{1-q^k}\right)^2  \\
&=\sum_{m,n\geq1} \frac{q^{m+n}(q^{m+n}-1)+q^{m+n}}{(1-q^m)(1-q^n)(1-q^{m+n})} + \left(\sum_{k\geq1}\frac{q^k}{1-q^k}\right)^2  \\
&=\sum_{m,n\geq1} \frac{q^{m+n}}{(1-q^m)(1-q^n)(1-q^{m+n})}.
\end{align*}

\noindent
On the other hand,
\begin{align*}
{\sum_{k\geq1}\sum_{\ell>k} \frac{k q^{\ell}}{(1-q^k)(1-q^{\ell})} } 
&= \sum_{m,n\geq1} \frac{mq^{m+n}}{(1-q^m)(1-q^{m+n})}  = \sum_{i,j,m,n\geq1} mq^{(i+j-1)m+ni} \\
&=\sum_{i,j\geq1} \frac{q^{2i+j-1}}{(1-q^i)(1-q^{i+j-1})^2}   \\
&=\sum_{m,n\geq1} \frac{q^{2m+n}}{(1-q^m)(1-q^{m+n})^2} + \sum_{i\geq1} \frac{q^{2i}}{(1-q^i)^3}.
\end{align*}

\noindent
We compute the difference
\begin{align*}
&\sum_{m,n\geq1} \frac{q^{m+n}}{(1-q^m)(1-q^n)(1-q^{m+n})}
- \sum_{m,n\geq1} \frac{q^{2m+n}}{(1-q^m)(1-q^{m+n})^2} \\
=& \sum_{m,n\geq1}  \frac{q^{m+n}}{(1-q^m)(1-q^{m+n})} \left\{\frac1{1-q^n}-\frac{q^m}{1-q^{m+n}} \right\}  
= \sum_{m,n\geq1} \frac{q^{m+n}}{(1-q^n)(1-q^{m+n})^2}.
\end{align*}

\smallskip
\noindent
The task now reduces to proving the identity
\begin{align} \label{identity1}
\sum_{m,n\geq1} \frac{q^{m+n}}{(1-q^n)(1-q^{m+n})^2}= \sum_{i\geq1} \frac{q^{2i}}{(1-q^i)^3}.
\end{align}

\noindent
Using Bell's identity (see \cite[page 117]{Andrews}), we obtain
\begin{align}\notag
\sum_{m,n\geq1} \frac{q^{m+n}}{(1-q^n)(1-q^{m+n})^2}
&=\sum_{n\geq1}\frac1{1-q^n}\sum_{k>n}\frac{q^k}{(1-q^k)^2}  \\
\notag &= \sum_{k\geq1} \frac{q^k}{(1-q^k)^2}\sum_{n=1}^k \frac1{1-q^n} - \sum_{k\geq1} \frac{q^k}{(1-q^k)^3} \\ \notag
&= \sum_{m\geq1} \frac{m^2q^m}{1-q^m}  - \sum_{k\geq1} \frac{q^k}{(1-q^k)^3} \\ \label{cube1}
&= \sum_{m\geq1} \frac{m^2q^m}{1-q^m}  - \sum_{m\geq1} \frac{(m^2+m)q^m}{2(1-q^m)} \\ \notag
&=\sum_{m\geq1} \frac{(m^2-m)q^m}{2(1-q^m)}\\ \label{cube2} & = \sum_{i\geq1} \frac{q^{2i}}{(1-q^i)^3}.
\end{align}
We used the elementary fact $\frac1{(1-x)^3}=\sum_{m\geq0}\frac{m(m-1)}2x^{m-2}$ to obtain \eqref{cube1} and also (applied in reverse) to arrive to \eqref{cube2}.
\end{proof}

\begin{theorem} \label{Th-f1-f3} We have $f_1(q)=f_3(q)$.
\end{theorem}
\begin{proof} 
Using algebraic manipulations, we obtain
\begin{align*} f_3(q) & = - \sum_{k\geq1}\sum_{\ell>k} \frac{2k q^{\ell}}{1-q^{\ell}} + \sum_{k\geq1}\sum_{\ell>k} \frac{2k q^{\ell}}{(1-q^k)(1-q^{\ell})}   \\
&= - \sum_{\ell\geq1} \frac{q^{\ell}}{1-q^{\ell}} \sum_{k=1}^{l-1} 2k + { \sum_{k\geq1}\sum_{\ell>k} \frac{2k q^{\ell}}{(1-q^k)(1-q^{\ell})} }   \\
&= - \sum_{\ell\geq1} \frac{(\ell^2-\ell)q^{\ell}}{1-q^{\ell}} + {\sum_{k\geq1} \frac{2k^2 q^k}{1-q^k} - \sum_{k\geq1} \frac{2kq^k}{(1-q^k)^2} } \\
&=  - \sum_{\ell\geq1} \frac{(\ell^2-\ell)q^{\ell}}{1-q^{\ell}} + \sum_{k\geq1} \frac{2k^2 q^k}{1-q^k} - \sum_{k\geq1} \frac{2kq^k}{1-q^k}  
-  \sum_{k\geq1} \frac{2kq^{2k}}{(1-q^k)^2}   \\
&= \sum_{\ell\geq1} \frac{(\ell^2-\ell)q^{\ell}}{1-q^{\ell}}  -  \sum_{k\geq1} \frac{2kq^{2k}}{(1-q^k)^2} =f_1(q).
\end{align*}
The third equality above follows from Lemma \ref{lemma_f1_f3}.
\end{proof}

We end this section with an alternative proof of Theorem \ref{Th-f1-f3}. For this purpose, we first define an operator and prove two lemmas. 

Define the \emph{grading-type} operator 
$$\mathcal{L}:=\sum_{k\geq1} (k\partial_{\alpha_k}^2+k^2\partial_{\alpha_k}+2k\partial_{\alpha_k}\sum_{\ell>k}\partial_{\alpha_{\ell}}).$$ 
The goal is to apply the operator $\mathcal{L}$ on the \emph{``generalized" $\eta$-function}
$$\frac1{\eta(\pmb{\alpha})} =\prod_{j\geq1} \frac1{1-e^{-\alpha_j}}, \qquad \text{where $\pmb{\alpha}=(\alpha_1,\alpha_2,\dots)$}. $$

\begin{lemma} \label{altlemma1} The following identity holds. 
\begin{align*}   \left. \mathcal{L} \left(\frac1{\eta({\pmb{\alpha}})}\right) \right\vert_{\substack{\alpha_j=j, \\ j\geq 1}}
& = \sum_{\lambda\in \mathcal{P}} e^{-\vert\lambda\vert}\sum_{k\geq1} \lambda_k(\lambda_k-2k+1).
\end{align*}
\end{lemma}

\begin{proof} We begin by expressing $\eta$ as a counting function:
\begin{align*}  \frac1{\eta(\pmb{\alpha})} & = \prod_{j\geq1} \sum_{i\geq0} e^{-i\alpha_j} = \sum_{m_1, m_2, \dots\geq0} e^{-\sum_{j\geq1} \alpha_j m_j}.
\end{align*}
Using the identification  $(m_1, m_2, \dots) \longleftrightarrow \lambda=(\lambda_1, \lambda_2, \dots)$ with  $m_j=\lambda_j-\lambda_{j-1}$,  we obtain
\begin{align*} \frac1{\eta(\pmb{\alpha})} & = \sum_{\lambda\in\mathcal{P}} e^{-\sum_{j\geq1} \alpha_j(\lambda_j-\lambda_{j+1})}
=  \sum_{\lambda\in\mathcal{P}} e^{-\sum_{j\geq1} \beta_j, \lambda_j},
\end{align*} 
where $\beta_k=\alpha_k-\alpha_{k-1}$ (assuming $\alpha_0=0$). This results in $\partial_{\alpha_k} = \partial_{\beta_k}-\partial_{\beta_{k+1}}$. Hence,
\begin{align*}  \mathcal{L} \left(\frac1{\eta}\right)
 &= \sum_{k\geq1} \left[k(\partial_{\beta_k}-\partial_{\beta_{k+1}})^2  + k^2(\partial_{\beta_k}-\partial_{\beta_{k+1}}) +2k(\partial_{\beta_k}-\partial_{\beta_{k+1}})
\sum_{\ell>k} (\partial_{\beta_{\ell}}-\partial_{\beta_{\ell+1}}) \right] 
e^{-\sum\beta_j \lambda_j}\\
&=  \sum_{k\geq1} \left[\partial_{\beta_k}^2+(2k-1)\partial_{\beta_k} \right] \sum_{\lambda\in\mathcal{P}} e^{-\sum_{j\geq1} \beta_j \lambda_j}  \\
&= \sum_{\lambda\in\mathcal{P}}  \left( \sum_{k\geq1} [(-\lambda_k)(-\lambda_k+2k-1)] \right)  e^{-\sum_{j\geq1} \beta_j \lambda_j}.
\end{align*}
Consequently, evaluating the last result at $\beta_j=1$, for all $j$, produces the desired claim.
\end{proof}

\begin{lemma} \label{altlemma2} If $q=e^{-1}$, then it holds that
\begin{align*}   \left. \mathcal{L} \left(\frac1{\eta}\right) \right\vert_{\substack{\alpha_j=j, \\ j\geq 1}}
& = \left.\frac1{\eta(\pmb{\alpha})} \right\vert_{\substack{\alpha_j=j, \\ j\geq 1}}\cdot \sum_{k\geq1} \left(\frac{k(k-1)q^k}{1-q^k} - \frac{2kq^{2k}}{(1-q^k)^2} -\sum_{k\geq1}\sum_{\ell>k} \frac{2kq^{k+\ell}}{(1-q^k)(1-q^{\ell})}\right).
\end{align*}
\end{lemma}

\begin{proof} We rely on the following elementary facts: if we let 
$f_k:=(1-e^{-\alpha_k})^{-1}$, then \begin{equation}\label{partial} \partial_{\alpha_k}f_k=f_k(1-f_k) \text{ \ \ and \ \ } \partial_{\alpha_k}^2f_k=f_k(1-f_k)(1-2f_k).\end{equation}  Although we do not require this at present, it is true, for all $t\geq 1$, that
$$ \partial_{\alpha_k}^tf_k=f_k(1-f_k)\sum_{j=1}^t(-1)^{j-1}j!\,S(t,j)\cdot f_k^{j-1},$$
where $S(t,k)$ are the \emph{Stirling numbers of the second kind}. 

In particular, from \eqref{partial}, we obtain that
$$\partial_{\alpha_k}(\eta^{-1})=(1-f_k) \eta^{-1}\qquad \text{and} \qquad \partial_{\alpha_k}^2(\eta^{-1})=(1-f_k)(1-2f_k) \eta^{-1}.$$
To finish the task, substitute these derivatives into 
$$\sum_k\left[k\partial_{\alpha_k}^2+k^2\partial_{\alpha_k}+2k\partial_{\alpha_k}\sum_{\ell>k} \partial_{\alpha_{\ell}}\right]\eta^{-1}$$ 
and then carry out the evaluations at $\alpha_j=j$ with $q=e^{-1}$.
\end{proof}

Now we can give an alternative proof of Theorem \ref{Th-f1-f3}.

\begin{proof}[Second proof of Theorem \ref{Th-f1-f3}] We show that $f_1(q)-f_3(q)=0$. Let  $$\kappa_{\lambda}:=\sum_{k\geq1} \lambda_k(\lambda_k-2k+1).$$ From Lemmas ~\ref{altlemma1} and ~\ref{altlemma2}, it follows that
$$ \left.\frac1{\eta(\pmb{\alpha})} \right\vert_{\substack{\alpha_j=j\\ j\geq1}}\cdot
\sum_{k\geq1} \left(\frac{k(k-1)q^k}{1-q^k} - \frac{2kq^{2k}}{(1-q^k)^2} -\sum_{k\geq1}\sum_{\ell>k} \frac{2kq^{k+\ell}}{(1-q^k)(1-q^{\ell})}\right)
=\sum_{n\geq1} q^n \sum_{\lambda\vdash n} \kappa_{\lambda}.$$
By using the Frobenius symbol $\lambda=(m_1,\dots,m_r\, \vert \, n_1,\dots,n_r)$, we obtain 
$$\kappa_{\lambda}=\sum_{i=1}^r [m_i(m_i+1)-n_i(n_i+1)].$$
Clearly $\kappa_{\lambda}= - \kappa_{\lambda'}$ where $\lambda'$ is the conjugate partition of 
$\lambda$. This forces the vanishing of $\sum q^n\sum\kappa_{\lambda}$, which completes the proof.
\end{proof}

\section{A related identity}

While studying the three Lambert series of section \ref{S1} and the relationship between them, we have discovered the identity in Theorem \ref{conj2} below. 
We denote the Kronecker symbol by $(\frac{p}{\cdot})$.

\begin{theorem} \label{conj2}   We have the identity
$$\sum_{n\geq1} \frac{(-1)^{n-1}  q^{\binom{n+1}2} (q;q)_n(q;q)_{n-1}}{(q;q)_{2n}} = \sum_{n\geq1} \frac{(\frac{5}n)\, q^n}{1-q^n}.$$
\end{theorem}

We first prove a few useful lemmas.

\begin{lemma} \label{lemma1} We have the identity
\begin{align}\label{WZ-lemma}
\sum_{j=1}^n\frac{(-1)^{j-1}q^{\binom{n-j}2}(q;q)_{n+j}(q;q)_{j-1}}{(q;q)_{n-j}(q;q)_{2j}}
&= q^{\binom{n}2} 
\left\{ \sum _{j=0}^{\lfloor\frac{n}2\rfloor}\sum_{i=-j+1}^j q^{j^2-i^2} +  
\sum _{j=0}^{\lfloor\frac{n-1}2\rfloor}\sum_{i=-j}^j q^{j^2+j-i^2-i}     \right\}.   
\end{align}
\end{lemma}

\begin{proof} We invoke  the Wilf-Zeilberger (WZ) method \cite{WZ}, for the pair
\begin{align*} F(n,j):&=\frac{(-1)^{j-1}q^{\binom{n-j}2}(q;q)_{n+j}(q;q)_{j-1}}{(q;)_{n-j}(q;q)_{2j}}    \\ 
G(n,j):&=\frac{(-1)^{j-1}q^{\binom{n+3-j}2}(q;q)_{n+j+1}(q;q)_j(1-q^{2n+4})}{(q;q)_{n+2-j}(q;q)_{2j}}. 
\end{align*}
Zeilberger's algorithm  furnishes the recurrence
$$F(n+3,j)-q^{n+2}F(n+2,j)-q^{3n+5}F(n+1,j)+q^{4n+5}F(n,j) = G(n,j)-G(n,j-1).$$
Summing the last equation over integers $1\leq j\leq n+3$ we arrive at
\begin{align} \label{recurrence1} 
f(n+3)-q^{n+2}f(n+2)-q^{3n+5}f(n+1)+q^{4n+5}f(n) = q^{\binom{n+3}2}(1+q^{n+2}),
\end{align}
where $f(n):=\sum_{j=1}^nF(n,j)$. On the other hand, if we let 
$$g(n):= q^{\binom{n}2} 
\left\{ \sum _{j=0}^{\lfloor\frac{n}2\rfloor}\sum_{i=-j+1}^j q^{j^2-i^2} +  
\sum _{j=0}^{\lfloor\frac{n-1}2\rfloor}\sum_{i=-j}^j q^{j^2+j-i^2-i} \right\}$$
then a routine calculation verifies that $g(n)$ satisfies the recurrence \eqref{recurrence1}. The proof follows immediately after checking the three initial conditions.   We omit the details.
\end{proof}

\smallskip
\noindent
With the notation
$$aa(n)=(-1)^{n-1} q^{\binom{n}2} 
\left\{ \sum _{j=0}^{\lfloor\frac{n}2\rfloor}\sum_{i=-j+1}^j q^{j^2-i^2} +  
\sum _{j=0}^{\lfloor\frac{n-1}2\rfloor}\sum_{i=-j}^j q^{j^2+j-i^2-i} \right\},$$ we have the following consequence of Lemma \ref{lemma1}.

\begin{corollary} \label{cor1} Let $aa1(n):=aa(n)+q^{n-1}aa(n-1)$. For all   integers $n\geq0$ we have 
$$aa1(2n)=-q^{3n^2-n}\sum_{j=-n+1}^n q^{-j^2} \qquad \text{and} \qquad
aa1(2n+1)=q^{3n^2+2n}\sum_{j=-n}^n q^{-j^2-j}.$$
Furthermore, 
$$aa(n)=\sum_{j=0}^n (-1)^jq^{nj-\binom{j+1}2} aa1(n-j).$$
\end{corollary}

\begin{proof} Multiplying both sides of \eqref{WZ-lemma} by $(-1)^{n-1}$,  the parity-dependent assertions follow after straightforward manipulations.

\smallskip
\noindent
Use $aa(n)=aa1(n)-q^{n-1}aa(n-1)$ and induction, on $n$, to prove the last claim.
\end{proof}

\begin{lemma} \label{lemma2} We have the identity
\begin{equation} \label{lemma2_eq} \sum_{n\geq0}\sum_{j=-n+1}^n   q^{5n^2-j^2} + \sum_{n\geq0}  \sum_{j=-n}^n q^{5n^2-j^2+5n-j+1}
=\sum_{m\geq1} \sum_{d\vert m} \left(\frac{5}d\right)q^m. \end{equation}
\end{lemma}

\begin{proof} Consider the real quadratic field $\mathbb{Q}(\sqrt5)$ and the corresponding ring of integers $\mathbb{Z}(\tau)=\{m+n\tau\vert m,n\in\mathbb{Z}\}$ where $\tau=\frac{1+\sqrt{5}}2$. The associated Dedekind zeta function is given by
\begin{equation}\label{zeta} \zeta_{\mathbb{Q}(\tau)}(s)=\sum_{\mathfrak{a}\subset\mathbb{Z}[\tau]}\frac1{[\mathbb{Z}[\tau]:\mathfrak{a}]^s}=\sum_{m\geq1} \frac{c(m)}{m^s}\end{equation}
where $\mathfrak{a}=\alpha\mathbb{Z}[\tau]$, $\alpha=a+b\sqrt 5$,  runs through the nonzero ideals of $\mathbb{Z}[\tau]$ and $$[\mathbb{Z}[\tau]:\mathfrak{a}]=\vert\alpha\vert=\vert a^2-5b^2\vert.$$
Recall the decomposition 
\cite[Part II, Section 11]{Zagier}
$$\zeta_{\mathbb{Q}(\tau)}(s)=\zeta(s)\cdot L(s,\chi)$$ with the Dirichlet character $\chi(n)=\left(\frac5n\right)$. 
\smallskip

The coefficient $c(m)$ in  \eqref{zeta}   gives the number of times the  natural number $m$ occurs as the norm of an ideal in $\mathbb{Z}[\tau]$ (i.e. the number of non-equivalent representations of $m$ by a quadratic form). Since
$\sum_n \frac1{n^s}\sum_r \frac{\chi(r)}{r^s}=\sum_n \sum_r\frac{\chi(r)}{(nr)^s}$,  we have  $$c(m)=\sum_{d\vert m} \chi(d)=\sum_{d\vert m} \left(\frac5d\right).$$

\smallskip
\noindent
Now, the exponents of $q$ on the left hand side of \eqref{lemma2_eq}  arise from   $-\vert j+n\sqrt5\vert=5n^2-j^2$ and the case of half-integers $-\vert \frac{2j+1}2+\frac{2n+1}2\sqrt5\vert=5n^2-j^2+5n-j+1$.  
 Finally, the ranges for the inner sums are directly justified by \cite[Lemma 3]{Andrews3}. This completes the proof. 
\end{proof}

\begin{remark}
  The Dirichlet series \eqref{zeta} has the following explicit expansion.
\begin{align*}
\zeta_{\mathbb{Q}(\tau)}(s)=\sum_{m\geq1} \frac{c(m)}{m^s}
&=\frac1{1-5^{-s}}\prod_{p\pm1\pmod 5}\frac1{(1-p^{-s})^2}\prod_{p\pm2\pmod 5}\frac1{1-p^{-2s}}  \\
&=\frac11+\frac1{4^s}+\frac1{5^s}+\frac1{9^s}+\frac2{11^s}+\cdots . 
\end{align*}
\end{remark}

\smallskip
\noindent
We now prove Theorem \ref{conj2}.

\begin{proof}[Proof of Theorem \ref{conj2}] Recall that we want to show that $$\sum_{n\geq1} \frac{(-1)^{n-1}  q^{\binom{n+1}2} (q;q)_n(q;q)_{n-1}}{(q;q)_{2n}} = \sum_{n\geq1} \frac{(\frac{5}n)\, q^n}{1-q^n}.$$
We will use a Bailey pair $(\alpha_n, \beta_n)$. Take $x=aq$ and $y\rightarrow\infty$ in \cite[eq'n (1.3)]{Slater} so that
$$\sum_{n\geq0} (-1)^na^nz^{-n}q^{\binom{n+1}2} (z;q)_n\beta_n
=\frac{(aq/z;q)_{\infty}}{(aq;q)_{\infty}}
\sum_{n\geq0} 
\frac{(-1)^nq^{\binom{n+1}2}a^nz^{-n}(z;q)_n\alpha_n}{(aq/z;q)_n}.$$
Now set $a=z=q$, and we obtain
$$\sum_{n\geq0} (-1)^n q^{\binom{n+1}2}(q;q)_n\beta_n
=(1-q)\sum_{n\geq0}(-1)^n q^{\binom{n+1}2}\alpha_n$$ 
with 
$$\beta_n=\sum_{r=0}^n \frac{\alpha_r}{(q;q)_{n-r}(q^2;q^2)_{n+r}}=
\sum_{r=0}^n \frac{(1-q)\alpha_r}{(q;q)_{n-r}(q;q)_{n+r+1}}= \sum_{r=0}^n \frac{\alpha'_r}{(q;q)_{n-r}(q;q)_{n+r+1}},$$
where $\alpha'_n=(1-q)\alpha_n$.  So,
$$\sum_{n\geq0} (-1)^n q^{\binom{n+1}2} (q;q)_n\beta_n
=\sum_{n\geq0} (-1)^n q^{\binom{n+1}2}\alpha'_n.$$
We take $\beta_0=0, \beta_n=\frac{(q;q)_{n-1}}{(q;q)_{2n}}$ to get
\begin{equation}\label{lhs-alpha'} \sum_{n\geq1} \frac{(-1)^{n-1} q^{\binom{n+1}2} (q;q)_n(q;q)_{n-1}}{(q;q)_{2n}}
=\sum_{n\geq1} (-1)^{n-1} q^{\binom{n+1}2} \alpha'_n.\end{equation}
Using the  inversion formula \cite[eq'n (4.1)]{Andrews4},  we obtain
$$\alpha'_n=(1-q^{2n+1})\sum_{j=1}^n\frac{(-1)^{n-j}q^{\binom{n-j}2}(q)_{n+j}(q)_{j-1}}{(q)_{n-j}(q)_{2j}}.$$
Lemma~\ref{lemma1} implies that  $\alpha'_n=(1-q^{2n+1}) aa(n)$. Thus, by Corollary~\ref{cor1},
\begin{equation}\label{alpha'}\alpha'_n=(1-q^{2n+1})aa(n)=(1-q^{2n+1}) \sum_{j=0}^n (-1)^j q^{nj-\binom{j+1}2} aa1(n-j).\end{equation}
From \eqref{lhs-alpha'} and \eqref{alpha'} we obtain 
\begin{align*} \sum_{n\geq1} \frac{(-1)^{n-1}q^{\binom{n+1}2} (q;q)_{n-1}(q;q)_n}{(q;q)_{2n}}
&=\sum_{n\geq0} (-1)^{n-1} q^{\binom{n+1}2}(1-q^{2n+1}) \sum_{j=0}^n (-1)^j q^{nj-\binom{j+1}2} aa1(n-j) \\
&=\sum_{n\geq0} (-1)^n q^{\binom{n+1}2} aa1(n) \sum_{j\geq0} (-1)^{j-1}q^{j^2+2jn}(1-q^{2n+2j+1}) \\
&= \sum_{n\geq0} (-1)^n q^{\binom{n+1}2} aa1(n)
\end{align*}
because the inner sum (in the last double sum) is identically $1$.

\smallskip
\noindent
Finally, using Corollary \ref{cor1}, we have
\begin{align*} \sum_{n\geq1} \frac{(-1)^{n-1} q^{\binom{n+1}2} (q;q)_{n-1}(q;q)_n}{(q;q)_{2n}} 
&=\sum_{n\geq0} q^{\binom{2n+1}2+3n^2-n}\sum_{j=-n+1}^nq^{-j^2}  \\
& + \sum_{n\geq0} q^{\binom{2n+2}2+3n^2+2n} \sum_{j=-n}^n q^{-j^2-j} \\
&=\sum_{n\geq0}\sum_{j=-n+1}^n   q^{5n^2-j^2} + \sum_{n\geq0}  \sum_{j=-n}^n q^{5n^2-j^2+5n-j+1}.
\end{align*}
 Applying Lemma~\ref{lemma2} completes the proof.
\end{proof}

We end this section with a result similar to that of Theorem \ref{conj2}.

\begin{theorem} \label{conj3} Suppose $\prod_{j\geq1} (1-q^j)=\sum_{m\geq0} a(m)\, q^m$. Then,
\begin{equation}\label{idconj3} \sum_{n\geq0} \frac{(-1)^n q^{\binom{n+1}2} (q;q)_n}{(q;q)_{2n}} = \sum_{m\geq0} (-1)^m\, \vert a(7m)\vert\, q^m.\end{equation}
\end{theorem}

\begin{proof} 
Consider the Bailey pair $(A_n,B_n)$ defined by
$$B_n=\frac1{(q;q)_{2n}} \qquad \text{and} \qquad B_n=\sum_{r=0}^n \frac{A_r}{(q^2;q^2)_{n+r}(q;q)_{n-r}}.$$

This can be inverted to produce $A_0=1$ and for $n>0$
$$A_n=\frac{1-q^{2n+1}}{1-q}  \sum_{j=0}^n \frac{(-1)^{n-j} q^{\binom{n-j}2} (q;q)_{n+j}}{(q;q)_{n-j}(q;q)_{2j}}.$$
Also, for $n>0$, introduce
$$C_n:= 
\begin{cases}    0  &\text{if $n\equiv1\pmod 3$} \\
(-q)^{\frac{n(2n-1)}3} \left( \frac{1-q^{2n+1}}{1-q}\right)  &\text{if $n\not\equiv 1 \pmod 3$}.
\end{cases}$$ Verifying that $$ A_n\cdot \frac{1-q}{1-q^{2n+1}} \ \ \text{and}\ \ C_n \cdot \frac{1-q}{1-q^{2n+1}}$$
 satisfy the recurrence 
$f(n+3)+q^{4n+5}f(n)=0$ shows that $A_n=C_n$ for $n>0$. We omit the details.

\smallskip
\noindent
Take $x=q^2, y=q$ and $z\rightarrow\infty$ in \cite[eq'n (1.3)]{Slater} so that
\begin{align*}
\sum_{n\geq0} \frac{(-1)^n q^{\binom{n+1}2} (q;q)_n}{(q;q)_{2n}} 
&=(1-q)\sum_{n\geq0} (-1)^nq^{\binom{n+1}2} A_n  \\
&= \sum_{n\geq0}  \left\{ q^{\frac{(3n+2)(7n+5)}2}(1-q^{6n+5}) + q^{\frac{n(21n+1)}2}(1-q^{6n+1})\right\}.
\end{align*}
It remains to check that the above expression  coincides with the right-hand side of \eqref{idconj3}. Again, we omit the routine calculation. 
\end{proof}

\section{A conjecture of Andrews, Dixit, Schultz, and Yee and related results}

The Lambert series $f_2(q)$ and $f_3(q)$ of Section \ref{S1} are reminiscent of the expression in the following conjecture.

\begin{conjecture} \cite[p. 24]{Andrews2} \label{George_conj} The following is an odd function of $q$, $|q|<1$:
$$Y(q):=\sum_{m,n\geq1} \frac{ (-q)^{2mn+m}}{(1+q^n)(1-q^{2m-1})}.$$
\end{conjecture}

While we are not able to prove the conjecture, we offer some ideas that are hopefully pointing in the right direction. 

\begin{proposition}
     \label{CA1} Define $\widetilde{Y}(q):=\frac{Y(q)}{-q}=\sum_{m,n\geq1}a(m,n)$, where  
$$a(m,n):=\frac{(-1)^{m-1}q^{2mn+m-1}}{(1+q^n)(1-q^{2m-1})}.$$ The coefficients of the $q$-series expansion of $a(m,n)$
are all either $-1$ or $0$ or $1$.
\end{proposition}

\begin{proof} Fix positive integers $m$ and $n$. Consider 
$$b(m,n):=\frac1{(1+q^n)(1-q^{2m-1})}.$$
Let $u:=\gcd(n,2m-1)$ and write $n=ua, 2m-1=ub$ with $\gcd(a,b)=1$. We make the substitution $z=q^u$. We want  to show that 
$$f(a,b):=\frac{1-z^a}{(1-z^{2a})(1-z^b)}=(1-z^a)\sum_{j,k\geq0}z^{2ak+jb}
=(1-z^a)\sum_{N\geq0} c_Nz^N$$
has only coefficients $0, 1$, or $-1$. It is well-known (via the \emph{Frobenius coin exchange problem}) that the largest $N$ for which $c_N=0$ is $2ab-2a-b$ while each coefficient is given by $c_N=\lfloor\frac{N}{2ab}\rfloor$ or $c_N=\lfloor\frac{N}{2ab}\rfloor+1$. 

\smallskip
\noindent
Since $a<2ab$, it follows that the coefficients of $z^N$ and $z^{N+a}$ differ at most by $1$. This completes the proof. 
\end{proof}

\begin{proposition} \label{step1} We have 
$$\widetilde Y(q)=\sum_{m,n\geq1} \frac{(-1)^{m-1} q^{2mn+m-1}}{(1+q^n)(1-q^{2m-1})}= \sum_{k\geq2} \frac{q^{k-1}}{1+q^{2k-1}} \sum_{n=1}^{k-1} \frac{q^n}{1+q^n}.$$
    \end{proposition}
    \begin{proof}
We consider the following direct calculation 
\begin{align*} 
\sum_{m,n\geq1} \frac{(-1)^{m-1} q^{2mn+m-1}}{(1+q^n)(1-q^{2m-1})}
&= - \sum_{m,n\geq1} \frac1{1+q^n}\sum_{j\geq0} q^{-j-1} (-1)^m q^{(2n+2j+1)m} \\
&=- \sum_{n\geq1} \frac1{1+q^n} \sum_{j\geq0}  q^{-j-1} \frac{-q^{2n+2j+1}}{1+q^{2n+2j+1}}  \\
&=\sum_{n\geq1} \frac1{1+q^n}  \sum_{j\geq0} \frac{q^{2n+j}}{1+q^{2n+2j+1}} \\
&=\sum_{n,i\geq1}  \frac{q^{2n+i-1}}{(1+q^n)(1+q^{2n+2i-1})}  \\
&=\sum_{n\geq1} \sum_{k\geq n+1}  \frac{q^{n+k-1}}{(1+q^n)(1+q^{2k-1})}   \\
&=\sum_{k\geq2} \frac{q^{k-1}}{1+q^{2k-1}} \sum_{n=1}^{k-1} \frac{q^n}{1+q^n}.
\end{align*}
\end{proof}

Related to the function $\widetilde Y(q)$, we prove the following result.

\begin{theorem} \label{relY} If $q=e^{-1}$, then it holds that 
$$\prod_{j\geq1} \frac1{1+q^j} \left(\sum_{k\geq1} \frac{q^k}{1+q^k} \sum_{\ell>k} \frac{q^{\ell-1}}{1+q^{2\ell-1}}\right)
= \sum_{\lambda\in \mathcal{P}}  (-1)^{\lambda_1} 
 \left( \sum_{\ell\geq1} q^{\vert\lambda\vert-\ell} (\lambda_{2\ell-1}-\lambda_{2\ell})(\lambda_1-\lambda_{\ell}) \right).$$
\end{theorem}

The identity of Theorem \ref{relY} follows from the lemmas below. 

Define the operator 
$$\mathcal{T}:=\sum_{k\geq1} \left[  \partial_{\alpha_k}\, \sum_{\ell>k}q^{-\ell} \partial_{\alpha_{2\ell-1}}\right],$$ 
and let it act on  the function 
$$\frac1{\widetilde{\eta}(\pmb{\alpha})} =\prod_{j\geq1} \frac1{1+e^{-\alpha_j}}, \qquad \text{where $\pmb{\alpha}=(\alpha_1,\alpha_2,\dots)$}. $$

\begin{lemma} \label{altlemma3} If $q=e^{-1}$, then it holds that
\begin{align*}   \left. \mathcal{T} \left(\frac1{\widetilde{\eta}}\right) \right\vert_{\substack{\alpha_j=j\\ j\geq1}}
& = \left.\frac1{\widetilde{\eta}(\pmb{\alpha})} \right\vert_{\substack{\alpha_j=j\\ j\geq1}}\cdot 
\sum_{k\geq1} \frac{q^k}{1+q^k} \sum_{\ell>k} \frac{q^{\ell-1}}{1+q^{2\ell-1}}
\end{align*}
\end{lemma}

\begin{proof} We rely on the following elementary facts: if we let $g_k:=(1+e^{-\alpha_k})^{-1}$, then $\partial_{\alpha_k}g_k=g_k(1-g_k)$.   
In particular, we obtain $\partial_{\alpha_k}(\widetilde{\eta}{-1})=(1-g_k) \widetilde{\eta}^{-1}$.
To finish the task, substitute these derivatives into $\sum_{k\geq1} \left[\partial_{\alpha_k}\, \sum_{\ell>k}q^{-\ell} \partial_{\alpha_{2\ell-1}}\right]\widetilde{\eta}^{-1}$ and then carry out the evaluations at $\alpha_j=j$ with $q=e^{-1}$.
\end{proof}

\begin{lemma} \label{altlemma4}  We have
\begin{align*}   
\left. \mathcal{T} \left(\frac1{\widetilde{\eta}}\right) \right\vert_{\substack{\alpha_j=j \\ j\geq1}}
& = \sum_{\lambda\in \mathcal{P}} (-1)^{\lambda_1} e^{-\vert\lambda\vert} 
 \left( \sum_{\ell\geq2} q^{-\ell} (\lambda_{2\ell-1}-\lambda_{2\ell})(\lambda_1-\lambda_{\ell}) \right).
\end{align*}
\end{lemma}

\begin{proof} We proceed as in Lemma \ref{altlemma1} and express $\widetilde{\eta}$ as 
\begin{align*}  \frac1{\widetilde{\eta}(\pmb{\alpha})} & = \prod_{j\geq1} \sum_{m_j\geq0} (-1)^{m_j} e^{-\alpha_j m_j} = 
\sum_{m_1, m_2, \dots\geq0} (-1)^{m_1+m_2+\cdots} e^{-\sum_{j\geq1} \alpha_j m_j}.
\end{align*}
Based on the identification  $(m_1, m_2, \dots) \longleftrightarrow \lambda=(\lambda_1, \lambda_2, \dots)$ with $m_j=\lambda_j-\lambda_{j+1}$, we obtain
\begin{align*} \frac1{\widetilde{\eta}(\pmb{\alpha})} & = \sum_{\lambda\in\mathcal{P}} (-1)^{\lambda_1-\lambda_2+\lambda_2-\lambda_3+\cdots} e^{-\sum_{j\geq1} \alpha_j(\lambda_j-\lambda_{j+1})}
=  \sum_{\lambda\in\mathcal{P}} (-1)^{\lambda_1} e^{-\sum_{j\geq1} \beta_j \lambda_j},
\end{align*}
where $\beta_k=\alpha_k-\alpha_{k-1}$ (assuming $\alpha_0=0$). This results in $\partial_{\alpha_k} = \partial_{\beta_k}-\partial_{\beta_{k+1}}$.  Hence,
\begin{align*}  \mathcal{T} \left(\frac1{\widetilde{\eta}}\right)
 &= \sum_{k\geq1} \left[ (\partial_{\beta_k}-\partial_{\beta_{k+1}}) 
 \sum_{\ell>k}q^{-\ell} (\partial_{\beta_{2\ell-1}}-\partial_{\beta_{2\ell}}) \right] e^{-\sum_{j\geq1} \beta_j \lambda_j}  \\
 &= \sum_{\ell\geq2} \left[ q^{-\ell} (\partial_{\beta_{2\ell-1}}-\partial_{\beta_{2\ell}})  \sum_{k=1}^{\ell-1} (\partial_{\beta_k}-\partial_{\beta_{k+1}}) \right] e^{-\sum_{j\geq1} \beta_j \lambda_j}  \\
 &= \sum_{\ell\geq2} \left[ q^{-\ell} (\partial_{\beta_{2\ell-1}}-\partial_{\beta_{2\ell}})   (\partial_{\beta_1}-\partial_{\beta_{\ell}}) \right] 
\sum_{\lambda\in\mathcal{P}} (-1)^{\lambda_1} e^{-\sum_{j\geq1} \beta_j \lambda_j}  \\
&= \sum_{\lambda\in\mathcal{P}}  (-1)^{\lambda_1} \left( \sum_{\ell\geq2} q^{-\ell}
[(\lambda_{2\ell-1}-\lambda_{2\ell})(\lambda_1-\lambda_{\ell})] \right)  e^{-\sum_{j\geq1} \beta_j \lambda_j}.
\end{align*}
Consequently, evaluating the last result at $\beta_j=1$ for all $j$, produces the desired claim. 
\end{proof}
The proof of Theorem \ref{relY} follows from Lemmas ~\ref{altlemma3} and ~\ref{altlemma4}.

Next, we prove a result about the parity of a $q$-series  closely related to  $\widetilde Y(q)$.

\begin{theorem} \label{other-even} We have the identity
$$\sum_{m,n \geq1} \frac{(-q)^{2mn+m-1}}{(1+q^{2n-1})(1-q^{2m-1})} = \sum_{n\geq1} \frac{q^{4n-2}}{(1-q^{4n-2})^2}.$$
\end{theorem}

To prove the theorem we use the following lemma.

\begin{lemma} \label{lemma11} We have the identities
\begin{align} 
\sum_{m,n \geq1} \frac{(-1)^{m-1}q^{2mn}[q^{m-1}+q^{-m}]}{(1+q^{2n-1})(1-q^{2m-1})} 
&  = q\left(\sum_{n\geq1} \frac{q^{n-1}}{1+q^{2n-1}} \right)^2, \label{id1} \\ 
\sum_{m,n \geq1} \frac{(-1)^{m-1}q^{2mn}[q^{-m}-q^{m-1}]}{(1+q^{2n-1})(1-q^{2m-1})} 
&=\sum_{n\geq1} \frac{q^{2n-1}}{(1+q^{2n-1})^2}.\label{id2}
\end{align}
\end{lemma}

\begin{proof} 
The same analysis as in the proof of Proposition \ref{step1} yields
$$\sum_{m,n \geq1} \frac{(-1)^{m-1} q^{2mn+m-1}}{(1+q^{2n-1})(1-q^{2m-1})}
=\sum_{n\geq1} \frac{q^{n-1}}{1+q^{2n-1}} \sum_{j=1}^{n-1} \frac{q^j}{1+q^{2j-1}} $$
as well as
\begin{align*}
\sum_{m,n\geq1} \frac{ (-1)^{m-1} q^{2mn-m}}{(1+q^{2n-1})(1-q^{2m-1})}  
&= \sum_{n\geq1} \frac{q^{n-1}}{1+q^{2n-1}} \sum_{j=1}^n \frac{q^j}{1+q^{2j-1}} \\
&= \sum_{n\geq1} \frac{q^{n-1}}{1+q^{2n-1}} \sum_{n=1}^{n-1} \frac{q^j}{1+q^{2j-1}}  + \sum_{j\geq1} \frac{q^{2j-1}}{(1+q^{2j-1})^2}.
\end{align*}

Identities \eqref{id1} and \eqref{id2} follow from the two identities above. For   identity \eqref{id1} we also use $$\sum_{n\geq1} \frac{q^{n-1}}{1+q^{2n-1}} \sum_{j=1}^n \frac{q^j}{1+q^{2j-1}}
=\sum_{j\geq1} \frac{q^{j-1}}{1+q^{2j-1}} \sum_{n\geq j} \frac{q^n}{1+q^{2n-1}}.$$
\end{proof}

We will now prove Theorem \ref{other-even}. First, we recall certain instrumental functions in the theory of modular forms. The three Eisenstein series $E_2(q), E_4(q),$ and $E_6(q)$ are given by
\begin{align*}
E_2(q)& :=1-24\sum_{n=1}^{\infty} \sigma_1(n)q^n,\\ 
E_4(q)& :=1+240\sum_{n=1}^{\infty}\sigma_3(n)q^n,\\  
E_6(q)& :=1-504\sum_{n=1}^{\infty} \sigma_5(n)q^n,
\end{align*}
 where $\sigma_v(n):=\sum_{d\mid n}d^v.$ We also use the Jacobi theta series
$$\theta_2(q)=\sum_{m\in\mathbb{Z}+\frac12}q^{m^2}, \qquad \theta_3(q)=\sum_{m\in\mathbb{Z}}q^{m^2} \qquad \text{and} \qquad
\theta_4(q)=\sum_{m\in\mathbb{Z}} (-1)^mq^{m^2}.$$
They satisfy the identity $\theta_3^4(q)=\theta_2^4(q)+\theta_4^4(q)$. The Eisenstein series and the Jacobi series are  linked by the following relations:
\begin{align*} 
\frac{D(\theta_2)}{\theta_2}&=\frac{E_2-\theta_2^4+5\theta_3^4}{24}, \qquad
\frac{D(\theta_3)}{\theta_3}=\frac{E_2+5\theta_2^4-\theta_3^4}{24}, \qquad
\frac{D(\theta_4)}{\theta_4}=\frac{E_2(q) - \theta_2^4-\theta_3^4}{24},
\end{align*} where the operator $D$ is defined by $D=q\frac{d}{dq}$.

\begin{proof}[Proof of Theorem \ref{other-even}] Using \eqref{id1} and \eqref{id2}, our task reduces to showing that
\begin{equation}\label{altY} q\left(\sum_{n\geq1} \frac{q^{n-1}}{1+q^{2n-1}}\right)^2
=\sum_{n\geq1} \frac{q^{2n-1}}{(1+q^{2n-1})^2} + 2\sum_{n\geq1} \frac{q^{4n-2}}{(1-q^{4n-2})^2}.\end{equation}
However, \eqref{altY} is a consequence of modularity, i.e.
\begin{align*}
q\left(\sum_{n\geq1} \frac{q^{n-1}}{1+q^{2n-1}}\right)^2=\frac1{16} \theta_2(q)
&=-\frac1{24}[E_2(q)-3E_2(q^2)+2E_2(q^4)], \\
\sum_{n\geq1} \frac{q^{2n-1}}{(1+q^{2n-1})^2}  = \frac12\frac{D\theta_3(q)}{\theta_3(q)}
&=-\frac1{24}[E_2(q)-5E_2(q^2)+4E_2(q^4)], \\
2\sum_{n\geq1} \frac{q^{4n-2}}{(1-q^{4n-2})^2} = -\frac{D\theta_4(q^2)}{\theta_4(q^2)}
&=-\frac1{12}[E_2(q^2)-E_2(q^4)].  
\end{align*}
A direct comparison concludes the proof. 
\end{proof}

\smallskip
\noindent
The next result is immediate from Theorem \ref{other-even}.

\begin{corollary} We have that
\begin{align*} 
[q^{2r-1}]\sum_{m,n \geq1} \frac{(-q)^{2mn+m-1}}{(1+q^{2n-1})(1-q^{2m-1})}&=0, \\
[q^{2r}]\sum_{m,n \geq1} \frac{(-q)^{2mn+m-1}}{(1+q^{2n-1})(1-q^{2m-1})}
&=\sum_{\substack{d\, \vert \, r \\ \text{$\frac{r}d$  is odd}}} d.
\end{align*}
\end{corollary}

We also have two related results. 
\begin{proposition} \label{kick1} For a positive integer $r$, let $t(r)$ denote that number of ways to write $r$ as a sum of two triangular numbers. Then,  It holds that
$$[q^{2r-1}] \sum_{n\geq1} \frac{q^n}{1+q^{2n-1}}
=t(n).$$
\end{proposition}

\begin{proof} This follows from  Ramanujan's ${}_1\psi_1$-summation 
$$\sum_{n\in\mathbb{Z}} \frac{z^n}{1-aq^n}
= \frac{(q;q)^2_{\infty} (az, \frac{q}{az}; q)_{\infty}} {(a,\frac{q}a,z, \frac{q}z; q)_{\infty}}.$$
Now take $q\rightarrow q^2$ then $z=q, a=-\frac1q$.  Thus
\begin{align*} \sum_{n\geq1} \frac{q^n}{1+q^{2n-1}} = \frac12 \sum_{n\in\mathbb{Z}} \frac{q^n}{1+q^{2n-1}} 
 & = \frac12 \frac{(q^2;q^2)^2_{\infty} (-1,-q^2;q^2)_{\infty}}{(-\frac1q, -q^3, q, q;q^2)_{\infty}}  \\
&= q\, \frac{(q^4;q^4)_{\infty}^2}{(q^2;q^4)_{\infty}^2}  \\
& = q\, \left(\sum_{n\geq0} q^{2\binom{n+1}2}\right)^2.
\end{align*}
The last equality follows from \cite[ eq'n (31.2)]{F}. 
\end{proof}

\begin{proposition} \label{kick2} For a positive integer $m$, we have that
$$\sum_{n\geq1} \frac{q^{2mn}}{1-q^n}
=\sum_{n\geq1} \frac{(-1)^{n-1}q^{(2m-1)n}q^{\binom{n+1}2}}{(1-q^n)(q^{2m};q^{2m})_n}.$$
\end{proposition}

\begin{proof} Using elementary fact that $\left.\frac{\partial}{\partial z}\right|_{z=1}(1-z)f(z)=-f(1)$ we obtain
\begin{align*}
\sum_{n\geq1}\frac{q^{2mn}}{1-q^n} & = -\left.\frac{\partial}{\partial z}\right|_{z=1} \sum_{n\geq0} \frac{q^{2mn}(z;q)_n}{(q;q)_n}
=  -\left.\frac{\partial}{\partial z}\right|_{z=1} \frac{(zq^{2m};q)_{\infty}}{(q^{2m};q)_{\infty}}.
\end{align*}
Since $\left.\frac{\partial}{\partial z}\right|_{z=1}(1-\frac1z)f(z)=f(1)$, we obtain
\begin{align*} \sum_{n\geq1} \frac{(-1)^{n-1}q^{(2m-1)n}q^{\binom{n+1}2}}{(1-q^n)(q^{2m};q^{2m})_n}
&= - \left.\frac{\partial}{\partial z}\right|_{z=1} \lim_{\tau\rightarrow 0} \,\,
{}_2\phi_1 \begin{pmatrix} \frac1z, \frac1{\tau}; q, q^{2m}\tau z   \\ q^{2m} \end{pmatrix}  \\
&= - \left.\frac{\partial}{\partial z}\right|_{z=1} \lim_{\tau\rightarrow 0} \,\,
\frac{(zq^{2m};q)_{\infty}(\tau q^{2m}; q)_{\infty}} {(q^{2m};q)_{\infty} (z\tau q^{2m}; q)_{\infty}}   \\
&= - \left.\frac{\partial}{\partial z}\right|_{z=1} \,\, \frac{(zq^{2m};q)_{\infty}}{(q^{2m};q)_{\infty}}.
\end{align*}
Thus the two expressions are identical.
\end{proof}

We conclude the section with two conjectures.

\begin{conjecture} \label{kick3} Let $a$ be a positive integer. Then, for each positive integer $n$, we have
$$[q^{n2^a}] \, \sum_{m,n\geq1} \frac{q^{mn2^a}}{(1+q^{n2^{a-1}})(1-q^{2m-1})}=\sigma_1(n),$$ where $\sigma_1(n)$ is the sum of divisors of $n$.
\end{conjecture}

\begin{conjecture} \label{kick4} If $r$ is a positive integer, then 
\begin{align*}
[q^{2r}] \sum_{m,n\geq1} \frac{q^{2mn}}{(1+q^{2n-1})(1-q^{2m-1})}&=[q^{2r}] \sum_{n\geq1} \frac{(n-1)q^n}{1+q^{2n-1}}. 
\end{align*} 
\end{conjecture}

\end{document}